\documentclass[a4paper,12pt]{amsart} 
\usepackage{amsmath,amssymb,euscript,verbatim,amsthm}
\usepackage[utf8]{inputenc}      
\usepackage[T1]{fontenc}
\usepackage[polish,english]{babel}
\usepackage{tikz, hyperref, aliascnt, color, mathtools}
\usetikzlibrary{patterns}
\usetikzlibrary{calc}
\usepackage{enumerate}
\DeclareMathSymbol{\shortminus}{\mathbin}{AMSa}{"39}

\makeatletter

\def\pgfsys@patternmatrix{1.0 0.0 0.0 1.0 0.0 0.0}

\def\pgfsys@declarepattern#1#2#3#4#5#6#7#8#9{%
  \pgf@xa=#2\relax%
  \pgf@ya=#3\relax%
  \pgf@xb=#4\relax%
  \pgf@yb=#5\relax%
  \pgf@xc=#6\relax%
  \pgf@yc=#7\relax%
  \pgf@sys@bp@correct\pgf@xa%
  \pgf@sys@bp@correct\pgf@ya%
  \pgf@sys@bp@correct\pgf@xb%
  \pgf@sys@bp@correct\pgf@yb%
  \pgf@sys@bp@correct\pgf@xc%
  \pgf@sys@bp@correct\pgf@yc%
  \immediate\pdfobj stream
  attr
  {
    /Type /Pattern
    /PatternType 1
    /PaintType \ifnum#9=0 2 \else 1 \fi
    /TilingType 1
    /BBox [\pgf@sys@tonumber\pgf@xa\space\pgf@sys@tonumber\pgf@ya\space\pgf@sys@tonumber\pgf@xb\space\pgf@sys@tonumber\pgf@yb]
    /XStep \pgf@sys@tonumber\pgf@xc\space
    /YStep \pgf@sys@tonumber\pgf@yc\space
    /Matrix [\pgfsys@patternmatrix]
    /Resources << >> 
  }
  {#8}%
  \pgfutil@addpdfresource@patterns{/pgfpat#1\space \the\pdflastobj\space 0 R}%
}

\def\pgf@sp{ }%
\def\pgftransformextractmatrix#1#2{%
\begingroup%
\pgftransformreset%
#2%
\xdef\pgf@tmp{\pgf@pt@aa\pgf@sp\pgf@pt@ab\pgf@sp\pgf@pt@ba\pgf@sp\pgf@pt@bb\pgf@sp\pgf@sys@tonumber\pgf@pt@x\pgf@sp\pgf@sys@tonumber\pgf@pt@y}%
\endgroup%
\let#1=\pgf@tmp}

\pgfdeclarepatternformonly[\patternangle]{rotated hatch}%
  {\pgfqpoint{-.1pt}{-1.pt}}{\pgfqpoint{5pt}{5pt}}
  {\pgfqpoint{5pt}{5pt}}
  {
    \pgfsetlinewidth{.5pt}
    \pgfpathmoveto{\pgfqpoint{-.1pt}{-.1pt}}
    \pgfpathlineto{\pgfqpoint{5pt}{5pt}}
    \pgfusepath{stroke}
  }

\tikzset{%
  pattern angle/.code={%
    \pgfmathparse{#1}\let\patternangle=\pgfmathresult
    \pgftransformextractmatrix\pgfsys@patternmatrix{\pgftransformrotate{\patternangle}}%
  },
  pattern angle=0
}

\newlength{\hatchspread}
\newlength{\hatchthickness}
\newlength{\hatchshift}
\newcommand{\hatchcolor}{}
\tikzset{hatchspread/.code={\setlength{\hatchspread}{#1}},
         hatchthickness/.code={\setlength{\hatchthickness}{#1}},
         hatchshift/.code={\setlength{\hatchshift}{#1}},
         hatchcolor/.code={\renewcommand{\hatchcolor}{#1}}}
\tikzset{hatchspread=3pt,
         hatchthickness=0.4pt,
         hatchshift=0pt,
         hatchcolor=black}
\pgfdeclarepatternformonly[\hatchspread,\hatchthickness,\hatchshift,\hatchcolor]
   {custom north west lines}
   {\pgfqpoint{\dimexpr-2\hatchthickness}{\dimexpr-2\hatchthickness}}
   {\pgfqpoint{\dimexpr\hatchspread+2\hatchthickness}{\dimexpr\hatchspread+2\hatchthickness}}
   {\pgfqpoint{\dimexpr\hatchspread}{\dimexpr\hatchspread}}
   {
    \pgfsetlinewidth{\hatchthickness}
    \pgfpathmoveto{\pgfqpoint{0pt}{\dimexpr\hatchspread+\hatchshift}}
    \pgfpathlineto{\pgfqpoint{\dimexpr\hatchspread+0.15pt+\hatchshift}{-0.15pt}}
    \ifdim \hatchshift > 0pt
      \pgfpathmoveto{\pgfqpoint{0pt}{\hatchshift}}
      \pgfpathlineto{\pgfqpoint{\dimexpr0.15pt+\hatchshift}{-0.15pt}}
    \fi
    \pgfsetstrokecolor{\hatchcolor}
    \pgfusepath{stroke}
   }

\definecolor{shade1}{rgb}{.2, .2, .2}
\definecolor{shade2}{rgb}{.5, .5, .5}
\definecolor{shade3}{rgb}{.8, .8, .8}
\definecolor{shade4}{rgb}{.95, .95, .95}
\usetikzlibrary{decorations.pathreplacing}
\usetikzlibrary{patterns}

\usetikzlibrary{arrows,shapes,positioning}
\usetikzlibrary{decorations.markings}
\tikzstyle arrowstyle=[scale=1]
\tikzstyle directed=[postaction={decorate,decoration={markings,
    mark=at position .65 with {\arrow[arrowstyle]{stealth}}}}]
\tikzstyle reverse directed=[postaction={decorate,decoration={markings,
    mark=at position .65 with {\arrowreversed[arrowstyle]{stealth};}}}]

\def\w{\omega}

\def\e{\varepsilon}

\def\Q{\mathbb{Q}} 
\def\Sb1{\mathbb{S}^1}
\def\T{\mathbb{T}}
\def\Z{\mathbb{Z}} 
\def\R{\mathbb{R}}

\def\i{{i,j}}

\def\S{S}
\def\T{T}

\newtheorem*{theorem*}{Theorem}
\newtheorem{theorem}{Theorem}
\newtheorem{lemma}[theorem]{Lemma}
\newtheorem{proposition}[theorem]{Proposition}
\newtheorem{corollary}[theorem]{Corollary}

\DeclareMathSymbol{\varnothing}{\mathord}{AMSb}{"3F} 
\begin{document}
\renewenvironment{proof}{\noindent {\bf Proof.}}{ \hfill\qed\\ }
\newenvironment{proofof}[1]{\noindent {\bf Proof of #1.}}{ \hfill\qed\\ }

\title{Chess billiards}

\def\curraddrname{{\itshape Address}}
\author{Arnaldo Nogueira}

\author{Serge Troubetzkoy}
\address{Aix Marseille Univ, CNRS, Centrale Marseille, I2M,  Marseille, France}

  \curraddr{I2M, Luminy\\ Case 907\\ F-13288 Marseille CEDEX 9\\ France}

 \email{arnaldo.nogueira@univ-amu.fr}
 \email{serge.troubetzkoy@univ-amu.fr}
 
 \date{\today}
 
 \thanks{
 The project leading to this publication has received funding from Excellence Initiative of Aix-Marseille University - A*MIDEX and Excellence Laboratory Archimedes LabEx (ANR-11-LABX-0033), French "Investissements d'Avenir" programmes.  The first author A.N. thanks the program CEFIPRA project No.\  5801-1/2017 for their support.}
\begin{abstract}  
We show the chess billiard map, which was introduced in \cite{HM} in order to study a generalization of the $n$-Queens problem in chess, is a circle homeomorphism.
 We give a survey of some of the known results on circle homeomorphisms, and apply them to this map.
 We prove a number of new results which  give  answers to some of the open questions posed in \cite{HM}.
\end{abstract} 
\maketitle

\section{Introduction}

The classical $n$-Queens problem asks in how many different ways $n$ mutually non-attacking queens can be placed on an $n \times n$ chess board.
In \cite{HM} Hanusa and Manhankali studied a generalization of this problem, and they introduced a dynamical system which they describe as ``reminiscent of billiards'' to aid their
study.  We will call this dynamical system the chess billiard, and in this article we will study it from a purely dynamical point of view.

The chess billiard is defined as follows, consider a convex planar domain $P$ and fix a direction $\theta_1$.
Foliating the plane by lines in this direction yields an involution of $T_1$ of the boundary $\partial P$ as follows: 
each line intersects $\partial P$ in either no point, one point, two points, or a segment.  In the cases when the  
intersection is non-empty, we define respective the map $T_1$ to be the identity, to exchange the two points, or 
to be the central symmetry of the segment 
about its center.  The chess billiard map is then the convolution of two such involutions defined by a  pair
of directions
$(\theta_1,\theta_2)$.
The chess billiard map turns out to be a circle homeomorphism.

The chess billiard map  was already introduced in various other articles
(without reference to chess),
Arnold mentions this map as his motivation for the study of KAM theory \cite{A}. The map was also studied in a special case   by Khmelev \cite{Kh}  (see Section \ref{secK}).

Our  article  has two purposes.
We first  collect  various well known results on circle homeomorphisms  and apply them to the chess billiard map
to deduce certain interesting results, in particular answering some questions posted in \cite{HM}.
Then we go on to prove new results about the chess billiard.  We prove some general results on periodic points, then 
we turn to the study of the chess billiard in a polygon. In particular we prove results about the chess billiard map in triangles, in centrally symmetric domains and in the square. Our results on the square are undoubtedly
the most interesting of the article.

The ultimate goal of the study of chess billiards is to understand  
 for which convex domains  and  directions the rotation number is rational and for which it is irrational.

\subsection{Structure of the article}
In Section \ref{sec-1.1} we give the formal definition of the chess billiard map and show that it is a circle homeomorphism
(Proposition \ref{prop-prop1})
and apply the theory of circle homeomorphisms, due to  Poncar\'e and Denjoy to them (Corollary \ref{cor-rotation}). 
Then in Section \ref{sec2} we go on to illustrate
this by analyzing the easiest case, the chess billiard map in the circle.
In Section \ref{sec3} we give a necessary and
sufficient condition for $\S$ to have a fixed point (Proposition \ref{prop-fix}), this allows us to understand
fixed points in strictly convex domains (Corollary \ref{cor1}) and to show that the chess billiard map is
purely periodic in an arbitrary triangle (Corollary \ref{cor2}). 
In this section we also  show that in a strictly convex domain the rotation number of the chess billiard
map achieves all values in $[0,1)$ (Theorem \ref{thm-all}).
In Section \ref{sec-square} we study the behavior of the chess billiard map in the square giving some
elements of an answer to questions 7.5 and 7.6 of \cite{HM}.
 We show that there exist directions for which the chess billiard map has no periodic orbit (Theorem \ref{thm-irrat}) and that
 for an
open dense set of directions $(\theta_1,\theta_2)$ the chess billiard has a periodic orbit (Theorem \ref{thm-sq}), however the set of directions which have neutral periodic orbits (i.e., are ``treacheries'' in the language of \cite{HM})
is small (Proposition \ref{prop-tre}).
We also give a necessary and sufficient condition
for  a vertex to be a periodic  point (Proposition \ref{prop-connection1})
and we give a sufficient condition
for the existence of neutral periodic orbits (Proposition \ref{prop-connection}). 
Finally in the short Section \ref{sec6} we study centrally symmetric domains.
Our results  give complete or partial answers to several questions raised in \cite{HM}.

\section{The chess billiard is a circle homeomorphism in disguise}\label{sec-1.1}

We begin by describing the chess billiard slightly more formally.
Consider a strictly convex planar domain $P$, and a foliation of $P$ by a family of nonoriented parallel lines (see Figure \ref{fig-circle}). 
The choice of families of lines is parametrized by $\theta \in  [0,\pi)$, i.e.,  the projective line. In some of
our arguments 
it will be convenient to parametrize $\theta$ by $[0,2\pi)$, this will be clear from the context. 
Since $P$ is convex we can parametrize its boundary $\partial P$ by arc length.
Since $P$ is strictly convex this foliation in direction $\theta_i$  induces a bijection ${T}_i : \partial P \to \partial P$ which is an orientation reversing homeomorphism with two fixed points. The boundary of a convex set is always rectifiable,
throughout the article we will express ${T}_i$ in the arc-length parametrization of $\partial P$, points
in $\partial P$ will be denoted $x$ or $y$.

\begin{figure}[h]
\begin{tikzpicture}[]
\draw[pattern=custom north west lines,hatchspread=7.627pt,hatchthickness=0.6pt]  (0,0) circle (50pt);
\draw [fill] (-1.25,1.25) circle [radius=2pt];
\draw [fill] (1.25,-1.25) circle [radius=2pt];
\draw [red,dashed,thick] (1.25,-1.25) -- (-1.25,1.25);

\node at (-1.4,1.4) {\tiny $x$};
\node at (1.5,-1.5) {\tiny ${T}_1x$};
\end{tikzpicture}
\caption{The circle foliated by a family of parallel lines, and the induced bijection.}\label{fig-circle}
\end{figure}
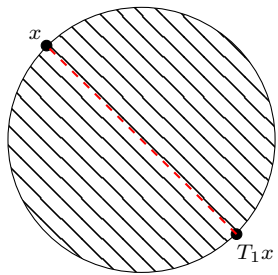

Let $Q := \partial P \times \{1,2\} =: Q_1 \cup Q_2$.  We fix the same orientation on   $Q_1$ and $Q_2$. 

Consider a pair of directions $(\theta_1,\theta_2)$, and the associated bijections 
${T}_i : Q_i \to Q_i$. The pair $(\theta_1,\theta_2)$ as well as each of its components will be referred to as
a direction.
Let  
\begin{equation*}
T=  \T_{\theta_1,\theta_2}: Q \to Q \ \  \text{ be defined by} \ \   T(x,i) = ({T}_{i}(x),i+1 \ (\text{mod } {2})).
\end{equation*}
The map ${T}$ is an orientation reversing homeomorphism.
We define the {\it chess billiard map} 
\begin{equation*}\label{def-map}
 S = \S_{\theta_1,\theta_2}: Q \to Q \ \  \text{by} \ \   S:= T^2.
\end{equation*}
The map $\T$ is an orientation reversing homeomorphism since
it is the composition of an orientation preserving homeomorphism $ F(i) := i+1  \ (\text{mod } {2})$ with an orientation reversing homeomorphism $T_i$. From here on we will not write the $\text{mod } 2$ when referring to the two copies of $Q$.
A degenerate case is when these two families coincide. 

The first proposition summarizes   several immediate properties of the map $S$.

\begin{proposition}\label{prop-prop1} Consider a convex domain $P$ and arbitrary $\theta_1,\theta_2$.
\begin{enumerate}
\item For each $i \in \{1,2\}$ the map $\S |_{Q_i}$ is an orientation preserving circle homeomorphism.
\item The sum of the rotation numbers of $\S |_{Q_1}$ and of $\S |_{Q_2}$ is $1$.
\item The map $\S$ is not topologically mixing. 
\end{enumerate}
\end{proposition}
Item (1) allows us to slightly misuse terminology and to refer to $S$ as a circle homeomorphism.
Item (2) allows us to slightly abuse the definition of rotation number and to refer to the rotation number $\rho(\S)$ of $\S$.

\begin{proof}
Item (1) is immediate since $S$ is the composition of   an orientation reversing homeomorphism $T$ with itself.
Item (2) follows immediately from the commutation relation $F S^n(x,i)  = S^{-n} F (x,i)$.
Item (3) follows by the construction, $\T(Q_i) = \T(Q_{i+1})$, and thus $\T$ is not topologically mixing.
\end{proof}

For completeness we give  the definition of rotation number introduced by Poincar\'e.
Consider the natural projection $\pi : \R \to Q_i$, it
provides a lift of $\S|_{Q_i}$ to a homeomorphism $\tilde{\S} : \R \to \R$ satisfying $\S \circ \pi = \pi \circ \tilde{\S}$. 
Consider $ \lim_{n \to \infty}  \frac{{\tilde S}^n(x) - x}{n}$, this limit always exists, and its value
does not depend on the point $x$ nor on the lift, thus we call it 
the {\em rotation number} $\rho(S)$  of ${S}$.

If $P$ is only convex,
the exact same definition works
 in the case that the direction of 
each interval in $\partial P$ is transverse to the two foliations; for example $P$ is a convex
polygon and neither of the two foliations is parallel to a side of $P$.
In this case we will call the direction $(\theta_1,\theta_2)$ \textit{exceptional} if either $\theta_1$ or $\theta_2$ is parallel 
to a side of $P$.
We can extend the definition to the exceptional case in the following way: we  
 define $T_i$ on any line segment in $Q_i$ to be the central symmetry with respect to the center of 
this segment  (see Figure \ref{fig-square}).  Again, the map $T$ is  an orientation reversing homeomorphism.
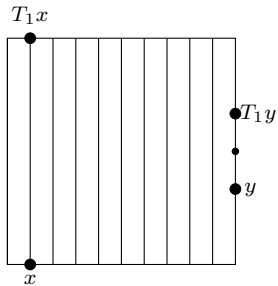
\begin{figure}[h]
\begin{tikzpicture}[scale=1]

\draw[] (4,0) rectangle + (3,3);
\foreach \i in {0,0.3,...,3}
\draw[] (4 + \i, 0) --(4 + \i,3);
\draw [fill] (4.3,0) circle [radius=2pt];
\draw [fill] (4.3,3) circle [radius=2pt];

\node at (4.3,-0.2) {\tiny $x$};
\node at (4.3,3.3){\tiny $T_1x$};

\draw [fill] (7,1) circle [radius=2pt];
\draw [fill] (7,2) circle [radius=2pt];

\node at (7.2,1) {\tiny $y$};
\node at (7.3,2){\tiny $T_1y$};
\draw [fill] (7,1.5) circle [radius=1.2pt];
\end{tikzpicture}
\caption{The square foliated by a family of parallel lines which are parallel to a side, and the induced bijection.}\label{fig-square}
\end{figure}

Our chess billiard map is  defined on all of $Q$, while the map  in \cite{HM} is  defined on
$\partial P \setminus \{ \text{the corners of } P\}$; the definitions agree where they are both defined.
Furthermore in \cite{HM}  the authors only consider points in $Q_i$ for which the
direction $\theta_i$ points towards the interior of $P$, while we allow  foliations  tangent to an interval in $\partial P$.

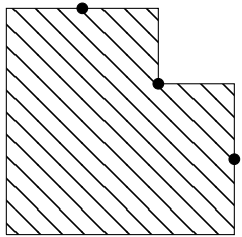
\begin{figure}[h]
\begin{tikzpicture}[]
\draw[pattern=custom north west lines,hatchspread=8.33pt,hatchthickness=0.6pt]  (0,0) -- (3,0) -- (3,2)-- (2,2) -- (2,3) -- (0,3) -- (0,0);
\draw [fill] (3,1) circle [radius=2pt];
\draw [fill] (1,3) circle [radius=2pt];
\draw [fill] (2,2) circle [radius=2pt];
\end{tikzpicture}
\caption{Any possible definition of the dynamics at the marked points will lead to a discontinuous map.}\label{fig.nonconvex}
\end{figure}

If $P$ is not convex, then for certain directions we have orbits which graze the boundary, any possible definition
of the dynamical system will lead to a discontinuous map  (see Figure \ref{fig.nonconvex}).
None the less, we can define a piecewise continuous chess map,
for example by defining the map by one sided continuity at these points (answering part of Question 7.8 of \cite{HM}).
Such a definition leaves our nice framework, and thus in this article we will not consider such domains, 
although their study is certainly interesting. In the case of a non-convex polygon, the resulting map is an 
affine interval exchange transformation, which
have been studied for the last decade or so, see \cite{BFG} and the references therein.

The study of the dynamics of orientation preserving circle homeomorphisms was initiated by Poincar\'e and developed by Denjoy, Herman, 
and others.  Some of the main results are summarized in 
the next corollary.  For the definitions and proof of the corollary
see for example \cite{KH}[Propositions 11.2.2 and 11.2.5, Theorem 11.2.9].

\begin{corollary}\label{cor-rotation}
If the rotation number of the chess billiard map $\S$ is rational, in reduced form $p/q$, then $\S$ has a periodic orbit, and all periodic orbits of $\S$ have period $q$.
If $\S$ has exactly one periodic orbit then every other point is heteroclinic under $\S^{q}$ to two points on the periodic orbit.  
These points are different if the period is greater than 1.  If $\S$ has more than one periodic orbit, then each nonperiodic point 
is heteroclinic under $\S^{q}$ to two points on different periodic orbits. 

Fix $i \in \{1,2\}$. If the rotation number of $\S$ is irrational then 
a) the $\omega$-limit set $\omega(x)$ is independent of $x \in Q_i$ and either $\w(x) =Q_i$ or it is a perfect and nowhere dense subset of $Q_i$, and
b) the map $\S$  is uniquely ergodic.
\end{corollary}

This corollary immediately gives answers to several of the questions posed in \cite{HM}:

\begin{itemize}
\item[7.1, 7.2:]  if we interpret the words ``predictable behavior'' as having zero topological entropy 
then the answer is that the chess billiard is always predictable since  circle homeomorphisms always have
zero entropy. 

\item[7.3:] for any (polygonal) board the directions which have a periodic orbit are those for which the rotation number
is rational.

\item[7.4:] the map is ergodic if and only if its rotation number is irrational (and since it is then uniquely ergodic this does not
depend on the starting point as asked in \cite{HM}).

\item[7.13:] all close points have the same behavior
in the sense of the corollary.  
\end{itemize}

Questions 7.3 and 7.4  are
quite general and the answers give quite a bit of information but are not definitive.

\section{The circle}\label{sec2}

The simplest case of chess billiards is that of the circle. 

\begin{proposition}
For the circle, the map $\S|_{Q_1}$ is the rotation of the circle by angle $2\alpha$ where $\alpha = \theta_2 - \theta_1$ (and $S|_{Q_2}$ is the rotation by angle $2\pi - 2\alpha$).
\end{proposition}

\begin{corollary}
If $\frac{\theta_2 - \theta_1}{\pi} = p/q \in \Q$ with $p \ge 0$, $q \ge 1$, $pgcd(p,q) = 1$, then all orbits are periodic with period $q$,
otherwise the map $\S$ is minimal and uniquely ergodic with respect to  the  length measure.
\end{corollary}

\begin{proof}
By rotational symmetry the behavior of the chess billiard map in the circle depends only on $\alpha := \theta_2 -  \theta_1$.
It is convenient to use complex coordinates, $\mathbb{S}^1= \{z \in  \mathbb{C}: |z|=1\}$.
Suppose $\theta_1 = \pi/2$ then we have $T_1(z) = \bar z$.
To compute $T_2(z)$ we rotate to make $\theta_2$ vertical, take the complex conjugate, and then rotate back: 
$z \to e^{i \alpha} z \to  \overline{e^{i \alpha} z }   \to e^{-i \alpha} \overline{e^{i \alpha} z} = e^{- i 2 \alpha} \bar z.$
Thus 
\begin{align*}  \S(z,1) & = (T_2 \circ T_1(z), 1) = (e^{- i 2 \alpha} z,1) \text{ and }\\
 \S(z,2) & = (T_1 \circ T_2(z),2) = (e^{i 2 \alpha} z,2).  \qedhere
\end{align*}
\end{proof}

There are 4 special points, the fixed points $\pm 1$ of $T_1$ and  the fixed points $e^{\pm i \theta_2}$ of $T_2$. 
They play a special role in the case of rational rotation number,  if $\frac{\alpha}{\pi} = p/q$ with $pgcd(p,q)=1$ and $q$  even, then $\S^p(\pm i)= \mp i$ and $\S^p(\pm 1) = \mp 1$ (orbits shown in red and blue
in  Figure \ref{fig-even}, while a generic orbit is shown in black), 
while if $q$ is odd, then the orbit of $\pm1$ arrives at $\pm  e^{i \theta_2}$ and then returns to itself
(orbits show in red and blue in Figure \ref{fig-odd}).
The role of these orbits will be investigated in the general setting in Section \ref{secConn}.

\begin{figure}[h]
\begin{minipage}[ht]{.5\linewidth}
\centering
\begin{tikzpicture}[scale=1]
\draw[]  (0,0) circle (1.5cm);

\draw [fill,red] (-1.5,0) circle [radius=2pt];
\draw [fill,red] (1.5,0) circle [radius=2pt];
\draw[red] (-1.5,0) --  (1.5,0);
\node at (-2.3,0) {\color{red} \tiny $1= \S^2 (1)$};
\node at (2.3,0) {\color{red} \tiny ${^-1}= \S (1)$};

\draw[blue] (0,-1.5) --  (0,1.5);
\draw [fill,blue] (0,-1.5) circle [radius=2pt];
\draw [fill,blue] (0,1.5) circle [radius=2pt];
\node at (0, -1.8) {\color{blue} \tiny $-i= \S(i)$};
\node at (0, 1.8) {\color{blue} \tiny ${i}= \S^2(i)$};

\draw [fill] (-1.25,0.8) circle [radius=2pt];
\draw [fill] (1.25,0.8) circle [radius=2pt];
\draw [fill] (-1.25,-0.8) circle [radius=2pt];
\draw [fill] (1.25,-0.8) circle [radius=2pt];

\draw (-1.25,0.8) -- (-1.25,-0.8) -- (1.25, -0.8) -- (1.25,0.8) --  (-1.25,0.8);
\node at (-2,-0.8) {\tiny $x = \S^2(x)$};
\node at (2,-0.8) {\tiny $T_1(\T(x))$};
\node at (-2.1,0.8) {\tiny $T_1(x)$};
\node at (1.8,0.8) {\tiny $\S(x)$};
\node at (0,-2.5)  {$Q_1$, rotation number $1/2$};
\end{tikzpicture}
\end{minipage}\nolinebreak
\begin{minipage}[ht]{.5\linewidth}
\centering
\begin{tikzpicture}[scale=1]
\draw[]  (0,0) circle (1.5cm);

\draw [fill,red] (-1.5,0) circle [radius=2pt];
\draw [fill,red] (1.5,0) circle [radius=2pt];
\draw[red] (-1.5,0) --  (1.5,0);
\node at (-2.3,0) {\color{red} \tiny $1= \S^2 (1)$};
\node at (2.3,0) {\color{red} \tiny ${^-1}= \S (1)$};

\draw[blue] (0,-1.5) --  (0,1.5);
\draw [fill,blue] (0,-1.5) circle [radius=2pt];
\draw [fill,blue] (0,1.5) circle [radius=2pt];
\node at (0, -1.8) {\color{blue} \tiny $-i= \S(i)$};
\node at (0, 1.8) {\color{blue} \tiny ${i}= \S^2(i)$};

\draw [fill] (-1.25,0.8) circle [radius=2pt];
\draw [fill] (1.25,0.8) circle [radius=2pt];
\draw [fill] (-1.25,-0.8) circle [radius=2pt];
\draw [fill] (1.25,-0.8) circle [radius=2pt];

\draw (-1.25,0.8) -- (-1.25,-0.8) -- (1.25, -0.8) -- (1.25,0.8) --  (-1.25,0.8);
\node at (-2,-0.8) {\tiny $x = \S^2(x)$};
\node at (1.8,-0.8) {\tiny $T_1(x)$};
\node at (-2.1,0.8) {\tiny $T_1( \S^2(x))$};
\node at (1.8,0.8) {\tiny $\S(x)$};
\node at (0,-2.5)  {$Q_2$, rotation number $1/2$};
\end{tikzpicture}
\end{minipage}
\caption{$\alpha = \pi/2$.
}\label{fig-even}
\end{figure}
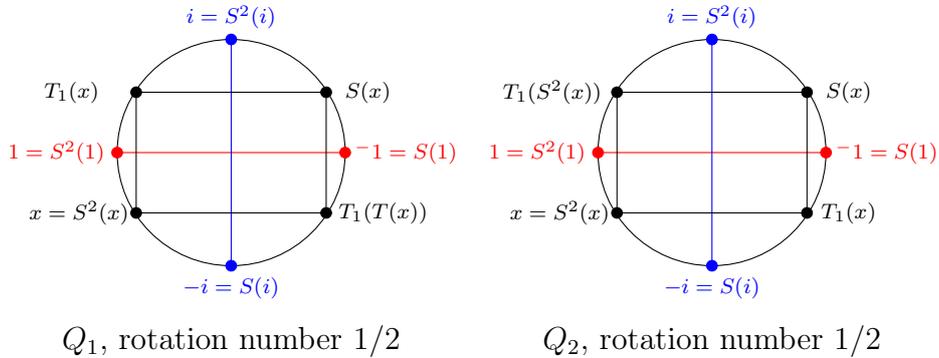


\begin{figure}[h]
\begin{minipage}[ht]{.5\linewidth}
\centering
\begin{tikzpicture}[scale=1]
\draw[]  (0,0) circle (1.5cm);

\draw [fill,red] (-1.5,0) circle [radius=2pt];
\draw [fill,red] (0.75,1.29) circle [radius=2pt];
\draw [fill,red](0.75,-1.29) circle [radius=2pt];

\node at (-2.3,0) {\color{red} \tiny $1= \S^3 (1)$};
\node at (0.9,1.7) {\color{red} \tiny $\S (1)$};
\node at (0.9,-1.7) {\color{red} \tiny $\S^2(1)$};

\draw[red] (-1.5,0) -- (0.75,1.29) -- (0.75,-1.29) ;


\draw [fill,blue] (1.5,0) circle [radius=2pt];
\draw [fill,blue] (-.75,1.29) circle [radius=2pt];
\draw [fill,blue](-0.75,-1.29) circle [radius=2pt];

\node at (2.3,0) {\color{blue} \tiny $1= \S^3 (1)$};
\node at (-0.9,1.7) {\color{blue} \tiny $\S (1)$};
\node at (-0.9,-1.7) {\color{blue} \tiny $\S^2(1)$};

\draw[blue] (1.5,0) -- (-0.75,1.29) -- (-0.75,-1.29) ;

\draw (-1.25,0.8) --(-1.25,- 0.8) ;
\draw (-1.25,0.8) -- (0.1,1.5);
\draw (0.1,1.5) -- (0.1,-1.5);
\draw (0.1,-1.5)  --(1.35,-0.65);

\draw (-1.25,-0.8) -- (1.35,0.65);
\draw (1.35,0.65) --(1.35,-0.65);

\draw [fill] (-1.25,0.8) circle [radius=2pt];
\draw [fill] (-1.25,-0.8) circle [radius=2pt];
\draw [fill]  (0.1,1.5) circle [radius=2pt];
\draw [fill] (0.1,-1.5) circle [radius=2pt];
\draw [fill]  (1.35,0.65) circle [radius=2pt];
\draw [fill] (1.35,-0.65) circle [radius=2pt];

\node at (-2,-0.8) {\tiny $x = \S^3(x)$};
\node at (2,-0.65) {\tiny $\S^2(x)$};
\node at (0.1,1.8) {\tiny  $\S^(x)$};

\node at (0,-2.5)  {$Q_1$, rotation number $1/3$};

\end{tikzpicture}
\end{minipage}\nolinebreak
\begin{minipage}[ht]{.5\linewidth}
\centering
\begin{tikzpicture}[scale=1]
\draw[]  (0,0) circle (1.5cm);

\draw [fill,red] (-1.5,0) circle [radius=2pt];
\draw [fill,red] (0.75,1.29) circle [radius=2pt];
\draw [fill,red](0.75,-1.29) circle [radius=2pt];

\node at (-2.3,0) {\color{red} \tiny $1= \S^3 (1)$};
\node at (0.9,1.7) {\color{red} \tiny $\S^2 (1)$};
\node at (0.9,-1.7) {\color{red} \tiny $\S(1)$};

\draw[red] (-1.5,0) -- (0.75,1.29) -- (0.75,-1.29) ;

\draw [fill,blue] (1.5,0) circle [radius=2pt];
\draw [fill,blue] (-.75,1.29) circle [radius=2pt];
\draw [fill,blue](-0.75,-1.29) circle [radius=2pt];

\node at (2.3,0) {\color{blue} \tiny $1= \S^3 (1)$};
\node at (-0.9,1.7) {\color{blue} \tiny $\S^2 (1)$};
\node at (-0.9,-1.7) {\color{blue} \tiny $\S(1)$};

\draw[blue] (1.5,0) -- (-0.75,1.29) -- (-0.75,-1.29) ;

\draw (-1.25,0.8) --(-1.25,- 0.8) ;
\draw (-1.25,0.8) -- (0.1,1.5);
\draw (0.1,1.5) -- (0.1,-1.5);
\draw (0.1,-1.5)  --(1.35,-0.65);

\draw (-1.25,-0.8) -- (1.35,0.65);
\draw (1.35,0.65) --(1.35,-0.65);

\draw [fill] (-1.25,0.8) circle [radius=2pt];
\draw [fill] (-1.25,-0.8) circle [radius=2pt];
\draw [fill]  (0.1,1.5) circle [radius=2pt];
\draw [fill] (0.1,-1.5) circle [radius=2pt];
\draw [fill]  (1.35,0.65) circle [radius=2pt];
\draw [fill] (1.35,-0.65) circle [radius=2pt];

\node at (-2,-0.8) {\tiny $x = \S^3(x)$};
\node at (2,-0.65) {\tiny $\S(x)$};
\node at (0.1,1.8) {\tiny  $\S^2(x)$};
\node at (0,-2.5)  {$Q_2$, rotation number $2/3$};
\end{tikzpicture}
\end{minipage}
\caption{$\alpha = \pi/3$.}
\label{fig-odd}
\end{figure}
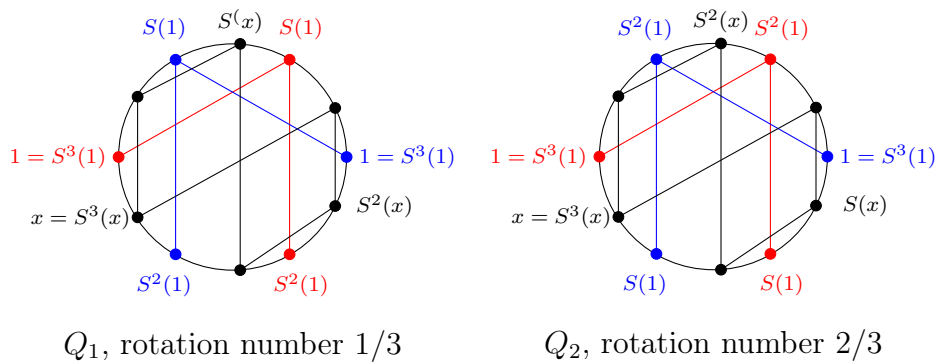

\section{Periodic orbits}\label{sec3}
\subsection{Fixed points of $\S$}
If the two foliations coincide, then $\S$ is the identity map. Now  
assume that the two foliations do not coincide, i.e., $\theta_1 \ne \theta_2$.

If $P$ is strictly convex then the map $T_i$  fixes the point $x$ if and only if a line of the foliation is a supporting 
line\footnote{A supporting line $L$ of a planar curve $C$  is a line that contains at least one point of $C$, and $C$ lies completely in one of the two closed half-planes defined by $L$.}  to $Q_i$ at $x$.
In the case $P$ is convex but not strictly  convex the necessary and sufficient condition is more complicated.
The point $x$ is a fixed point if and only if either  i) a line of the foliation is a supporting line  to $Q_i$ at $x$, and $x$ is isolated in this intersection;
or ii) a line of the foliation is tangent to $Q_i$ in an interval, and $x$ is the center point of this interval.

If a point $x \in \partial P$ is fixed by both $T_1$ and $T_2$, then it is a fixed point of $\S$ and thus the rotation number 
of $\S|{Q_i}$ is zero for $i \in \{1,2\}$. The converse is also true.  Suppose that $x \in Q_1$ is not fixed by $T_1$, then
since by assumption the two foliations are not parallel $\T_1(x) \ne x$. Combining this with Corollary \ref{cor-rotation} we have shown

\begin{proposition}\label{prop-fix} Suppose that  $\theta_1 \ne\theta_2$.
The map $\S_{\theta_1,\theta_2}$ has a fixed point if and only if there is a point $x \in \partial P$ such that the lines through $x$ in these two
directions are  supporting lines at $x$. Moreover, each fixed point is semi stable,  i.e., repelling, from one side, and attracting from the opposite  side.
\end{proposition}

\begin{corollary}\label{cor1}
A $C^1$ strictly convex domain can not have any fixed point (unless $\theta_1 = \theta_2$).
\end{corollary}

\begin{corollary}\label{cor2}
If $P$ is a triangle and $\theta_1,\theta_2$ are arbitrary, then $\S_{\theta_1,\theta_2}$
has a fixed point or a periodic point with period 2 or  3.
\end{corollary}

\begin{proof}
If $\theta_1 = \theta_2$ then every point is fixed by $\S$, thus we assume that they are not equal.

Next consider the case when $(\theta_1,\theta_2)$ is not exceptional. 
Consider the lines of the foliation in the direction $\theta_1$ 
which intersect $P$, the extremal ones are supporting lines which pass through two distinct vertices of $P$.
The same holds for $\theta_2$, since $P$ has only three vertices there is a vertex for which
both  directions must have  supporting lines, and thus a fixed point by Proposition \ref{prop-fix}.

Turning to the case when $(\theta_1,\theta_2)$ is exceptional, we begin by 
treating the case when  only one direction is parallel to a side, say $\theta_1$, it is also a supporting line of the vertex opposite to this side.
If $\theta_2$ is a supporting line of this vertex then again  applying Propostion \ref{prop-fix} we conclude that  this vertex is fixed (see
Figure \ref{fig-pear}).
Otherwise $\theta_2$ is a supporting line of the
two endpoints of the side parallel to $\theta_1$; these endpoints are fixed by $\overline{T}_2$ and exchanged by $\overline{T}_1$, thus they are
exchanged by $\S$.

\begin{figure}[h]
\begin{tikzpicture}[]
\draw[] (0,0) -- (5,0) -- (3,1) -- (0,0);
\draw[dotted,->] (0,0) -- (6,0);
\draw[dotted,->]  (5,0) -- (2,1.5);
 \node at (-0.25,-0.25) {\tiny $x = \S^3(x)$};
\node at (1.75,1.7) {\tiny $\theta_2$};
\node at (6.25,0) {\tiny $\theta_1$};
\node at (3.25,1.25) {\tiny $S(x)$};
 \node at (5.25,-0.25) {\tiny $\S^2(x)$};
\draw [fill] (5,0) circle [radius=2pt];
\draw [fill] (0,0) circle [radius=2pt];
\draw [fill] (3,1) circle [radius=2pt];
 \end{tikzpicture}
\caption{Period 3 orbit in a  triangle.}
\label{fig-period3}
\end{figure}
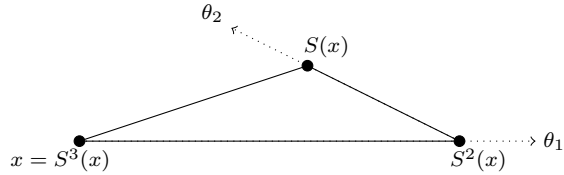
Finally in the case when both directions are parallel to a side, the map $\S$ cyclically exchanges the vertices of the triangle (Figure \ref{fig-period3}).
\end{proof}

If $P$ is a square and both foliation directions are in the same quadrant, then $\S$ has two fixed points, both semi-stable, (this is 
essentially contained in \cite{HM}, however they have not defined the dynamics at the two fixed points).
More generally, 
any convex polygon or even any convex table with a corner has an open set of pairs of directions for which $\S$ has a fixed point.  
Examples of strictly convex domains and of convex polygons with exactly one fixed point exist (see Figure \ref{fig-pear}). This point is repelling on one side and attracting on the other side.

\begin{figure}[h]
\begin{tikzpicture}[rotate=-90]
 \draw [domain=-180:0] plot ({cos(\x)}, {sin(\x)});
\draw [domain=-180:-270] plot ({cos(\x)}, {(1 - (\x+180)/100) *sin(\x)});
\draw [domain=0:90] plot ({cos(\x)}, { (1+\x/100) * sin(\x)});

\draw[blue] (-1,0) -- (1,0) -- (-.72,1) -- (.72,1) -- (-.33,1.6) -- (.33,1.6);

\draw[] (1,5) -- (-1,5) -- (-1.1,7.1) -- (1,5);
\draw[blue] (-0.5,5) -- (-0.5,6.5) -- (0.25,5) -- (0.25,5.75)  -- (0.625,5) -- (0.625,5.375);
 \end{tikzpicture}
\caption{A strictly convex domain and a triangle each having a single semi-attracting fixed point.}
\label{fig-pear}
\end{figure}
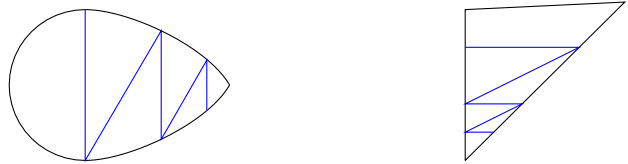

\subsection{Connections}\label{secConn}

We call an $\S$-orbit segment starting and ending at fixed points of the maps $T_i$ \textit{a  connection}.
We can think of a connection as a broken line, the \textit{length} of the connection is the number of 
segments in this line (so a fixed point has length $0$).
The following result generalizes what we showed for the circle and for fixed points.

\begin{proposition}\label{prop-connection2}
If there is a connection then the map $\S$  has a periodic point.
\end{proposition}

\begin{proof}
 Proposition \ref{prop-fix} is a special case of this result for fixed points.  
 Consider the set $F:= \{x: \text{ either } T_1(x) = x \text{ or } T_2(x) = x\}$.
 
 Suppose $x \in F$.  Then the $\S$ orbit which arrives at $x$ reverses its direction and 
retraces the orbit in the opposite direction.  Thus if it arrives at another point $y \in F$, then it is a periodic orbit.
\end{proof}

\subsection{The rotation number achieves all values}\label{sec2.1}

If the direction $\theta_2$ is not parallel to a segment in $\partial P$, then the map $(\theta_2, x) \mapsto T_2(x)$ is a monotone continuous function of  $\theta_2$ .  Thus the chess billiard map $S_{\theta_1,\theta_2}$ is a continuous function of $\theta_2$. Since the rotation number depends  continuously on the map 
we have shown 

\begin{proposition}\label{prop-cont}
 For each fixed $\theta_2$, if the direction $\theta_1$ is not parallel to a segment in $\partial P$, then \\
1) the rotation number map $\theta_1 \mapsto \rho(\T_{\theta_1,\theta_2})$ is a monotone
function of $\theta_1$ and \\
2) the  point $\theta_1$ is a point of continuity of the rotation number map.  
\end{proposition}

\begin{theorem}\label{thm-all}
Fix a strictly convex table and a direction $\theta_2$, then as we vary $\theta_1 \in [0,\pi)$ the rotation number  of $\S$ achieves all values in $[0,1)$.
\end{theorem}

\begin{proof} In the proof we think of $\theta_1$ and $\theta_2$ as  oriented vectors.
For  $\theta_2=\theta_1$ and  
for  $\theta_2=\theta_1 + \pi$ we have $\S=id$. In both cases  the rotation number is $0$. 
By Corollary \ref{cor1}, for fixed $\theta_1$  the rotation number is 
non zero for all $\theta_2 \not \in \{  \theta_1, \theta_1 + \pi\}$.  Furthermore the rotation number is monotonic
and continuous (by Proposition \ref{prop-cont})  in $\theta_2$.  Combining
these facts  implies that the rotation number varies from 0 to 1 as 
 $\theta_2$ varies from $\theta_1$ to $\theta_1 + \pi$, in the sense that $\displaystyle{ \lim_{\tiny \theta_2 \nearrow \theta_1 + \pi} \rho(S_{\theta_1,\theta_2} )= 1}$.
\end{proof}

\subsection{Khmelev result}\label{secK}

Khmelev  \cite{Kh} showed that if $P$ is convex and is sufficiently smooth  everywhere except one point
where the first derivative has a jump discontuity then the rotation number $\rho(\S)$ is rational 
for almost all values of $\theta_1,\theta_2$  (see his article for the precise smoothness assumptions).

\subsection{Periodic orbits in polygons}\label{sec5}

Fix $P$, and suppose that the rotation number associated to a pair of directions $\theta_1,\theta_2$ is rational, $p/q$ in reduced form.
Let $I$ be an interval contained in a side of $P$, perhaps degenerate to a point, such that $\S^q I = I$ and $\S^q J \ne J$ for any $J \supset I$;
we call $C(I) := \cup_{j =0}^{q-1} \S^jI$ a {\em periodic cylinder}.
By continuity a periodic cylinder $C(I)$ is always a closed set. In the case $I$ degenerates to a point, a periodic cylinder is 
simply a periodic orbit of period $q$.  
 If the interval $I $ is not degenerate  then we will call each $x \in I$ a {\em neutral periodic orbit} and $C(I)$ a
 {\em neutral cylinder}, 
in the language of \cite{HM} a neutral cylinder is called a {\em treachery} (see Example 5.3 of \cite{HM} 
to understand this connection).

\begin{proposition}
Suppose that $P$ is a convex polygon with $k$ sides, and $(\theta_1,\theta_2)$ is such that the rotation number $\rho(\S) = p/q$ is rational, then the number of periodic cylinders for $\S|_{Q_i}$ is at most $3k-4$ $(i =1,2)$.
\end{proposition}

\begin{proof}
Consider $P$ foliated by lines in the direction $\theta_1$.  There are 2 lines of this foliation which are supporting lines,
and (at most) $k - 2$ other lines which pass through a vertex of $P$. 
Consider the intersection of these lines with $\partial P$, this intersection consists of the $k$ corners plus (at most)
$k-2$ other points, so (at most) $2k -2$ points. 
These points partition $Q_1$ into (at most) $2k-2$ intervals on which 
$T_1$ is affine.  

The same construction yields  (at most) $2k-2$ points in $Q_2$ for the direction $\theta_2$.
Take the preimage $T_1^{-1}$ of these points,  yields  (at most) $2k-2$ points in $Q_1$, however  $k$
of these points (namely $T^{-1}_1$ of the vertices) are in the previously defined collection of points in $Q_1$. 
Thus in total we have (at most) $3k-4$ points in $Q_1$, which define (at most) $3k-4$ intervals on which $S$ is affine.
 Therefore
the map $\S^q|_{Q_i}$ has at most $q(3k-4)$ intervals of affinity. But each piece of affinity can  intersect the diagonal at most one time.\end{proof}

\section{The square}\label{sec-square}

Suppose the square is $[0,1]^2$.

\begin{theorem}\label{thm-irrat}
There exists a direction $(\theta_1,\theta_2)$ such that the chess billiard map $S_{\theta_1,\theta_2} $ in the square has an irrational rotation number (and thus is aperiodic).
\end{theorem}

\begin{proof} 
If $\theta_1 = \pi/4$ and $\theta_2 = 3\pi/4$ then all $S$-orbits in the square have period 2.
On the other hand  if $\tan{\theta'_1} = 1/3$ and $\tan{\theta'_2} = -2/3$ then a simple geometric exercise
shows that the orbit of the point $(1,1/2)$ is a period 3 orbit (see \cite{HM}, Figure 9).
Consider the line segment $L \subset \R^2$ with endpoints  $(\pi/4,3\pi/4)$ and $(\arctan(1/3),\arctan(-2/3))$.  The  function
$(\theta_1'',\theta_2'') \mapsto S_{\theta_1'',\theta_2''}$ is a continuous function when  $(\theta_1'',\theta_2'') \in L$ since $L$ does not intersect the set of exceptional directions.. Therefore
the rotation number 
$\rho(S_{\theta''_1,\theta''_2})$ is a continuous function of $(\theta''_1,\theta''_2)$, and thus it takes all values between $1/2$ and $1/3$.
\end{proof}

\begin{theorem}\label{thm-sq} The chess billiard map $S_{\theta_1,\theta_2} $ in the square 
has a periodic point for an open dense set of $(\theta_1,\theta_2) \in \Sb1 \times \Sb1$.
\end{theorem}

The proof uses another cross-section to the chess billiard flow, which relies on the symmetries of the square.
Throughout the proof we suppose that  
the directions $\theta_1$ and $\theta_2$ are not exceptional,  and furthermore we suppose that they
are in different quadrants, since if they are in the same quadrant 
the map has a fixed point.
It suffices to  treat the case $\theta_1 \in (0, \pi/2)$ and
$\theta_2  \in (\pi/2,\pi)$.

Let $D$ denote  the diagonal  $x+y =1$ of the square. We define a  map $F: D \to D$. We give two different
descriptions of this map.   
Start at a point in $D$, flow in the direction $\theta_1$ (towards the right) until we reach the boundary of the square, then flow in the direction 
$\theta_2$ until we return to the boundary of the square,
and finally again flow in the direction $\theta_1$ until we return to $D$.  The point we have returned to is in $D$, but we can be flowing either
to the right or to the left depending on if the flow in the direction $\theta_2$ had crossed the diagonal or not; 
if we are flowing to the right call this point $F(x)$ while if we are flowing to the left we apply a central symmetric to obtain $F(x)$.

\begin{figure}[h]
\begin{minipage}[ht]{.296\linewidth}
\begin{tikzpicture}[x=5cm,y=5cm];
     \draw[red,thick] (0,0) -- (0.75,0.435);
\foreach \i [count=\j from 0] in {-15}{
   \draw [pattern=rotated hatch, pattern angle=\i] 
     ({mod(\j,4)}, {floor(\j/4)}) rectangle ++(0.75,0.75);   }
 \end{tikzpicture}
\end{minipage}\nolinebreak
\begin{minipage}[ht]{.2956\linewidth}
   \begin{tikzpicture}[x=5cm,y=5cm];
   \draw[thick] (0.75,0) --(0,0.75);
   \draw[red,thick] (0,0.435) -- (0.184,0.566);
\node at (0.33,0.55) {\tiny $A'_2$};
\node at (0.6,0.25) {\tiny $A'_3$};
   \draw[red,thick] (0.44,0.31) --(0,0);
     \foreach \i [count=\j from 0] in {-10}{
   \draw [pattern=rotated hatch, pattern angle=\i] 
     ({mod(\j,4)} , {floor(\j/4)}) rectangle ++(0.75,0.75);   }
\end{tikzpicture}
\end{minipage}\nolinebreak
\begin{minipage}[ht]{.296\linewidth}
\begin{tikzpicture}[x=5cm,y=5cm];
\foreach \i [count=\j from 0] in {-15}{
   \draw [pattern=rotated hatch, pattern angle=\i] 
     ({mod(\j,4)}, {floor(\j/4)}) rectangle ++(0.75,0.75);   }
 \end{tikzpicture}
\end{minipage}\nolinebreak
\vspace{-0.6mm}
\begin{minipage}[ht]{.296\linewidth}
\begin{tikzpicture}[x=5cm,y=5cm];
\draw[thick] (0.75,0) --(0,0.75);
\draw[red,thick] (0.31,0.44) --(0.75,0.75);
   \draw[red,thick]  (0.75,0.315) -- (0.566,0.184);
\node at (0.1,0.55) {\tiny $A_2$};
\node at (0.34,0.31) {\tiny $A_3$};
\node at (0.58,0.08) {\tiny $A_1$};
\foreach \i [count=\j from 0] in {-10}{
   \draw [pattern=rotated hatch, pattern angle=\i] 
     ({mod(\j,4)}, {floor(\j/4)}) rectangle ++(0.75,0.75);   }
 \end{tikzpicture}
\end{minipage}\nolinebreak
\begin{minipage}[ht]{.2956\linewidth}
     \begin{tikzpicture}[x=5cm,y=5cm];  
     \draw[red,thick] (0,0) -- (0.75,0.435);
     \draw[red,thick] (0,0.315) -- (0.75,0.75);
      \foreach \i [count=\j from 0] in {-15}{
   \draw [pattern=rotated hatch, pattern angle=\i] 
     ({mod(\j,4)} , {floor(\j/4)}) rectangle ++(0.75,0.75);   }
\end{tikzpicture}
\end{minipage}\nolinebreak
\begin{minipage}[ht]{.296\linewidth}
\begin{tikzpicture}[x=5cm,y=5cm];
\draw[thick] (0.75,0) --(0,0.75);
     \draw[red,thick] (0,0.435) -- (0.184,0.566);
\node at (0.13,0.68) {\tiny $A'_1$};
\foreach \i [count=\j from 0] in {-10}{
   \draw [pattern=rotated hatch, pattern angle=\i] 
     ({mod(\j,4)}, {floor(\j/4)}) rectangle ++(0.75,0.75);   }
 \end{tikzpicture}
\end{minipage}

\caption{The map $F: D \to D$, case $\phi_2 \in (0,\pi/4)$}
\label{fig-unfold1}
\end{figure}
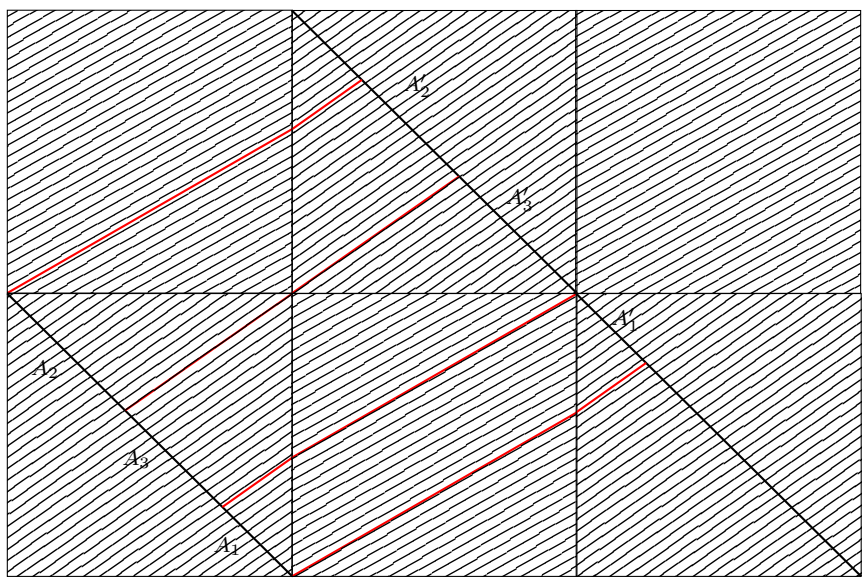

\begin{figure}[h]
\begin{minipage}[h]{.296\linewidth}
\begin{tikzpicture}[x=5cm,y=5cm];
\draw[thick] (0.75,0) --(0,0.75);
   \node at (0.65,0.21) {\tiny $A'_1$};
     \draw[red,thick]  (0.13,0) -- (0.495,0.255);
\foreach \i [count=\j from 0] in {-10}{
   \draw [pattern=rotated hatch, pattern angle=\i] 
     ({mod(\j,4)}, {floor(\j/4)}) rectangle ++(0.75,0.75);   }
 \end{tikzpicture}
\end{minipage}\nolinebreak
\begin{minipage}[ht]{.2956\linewidth}
     \begin{tikzpicture}[x=5cm,y=5cm];
      \foreach \i [count=\j from 0] in {35}{
   \draw [pattern=rotated hatch, pattern angle=\i] 
     ({mod(\j,4)} , {floor(\j/4)}) rectangle ++(0.75,0.75);   }
\end{tikzpicture}
\end{minipage}\nolinebreak

\vspace{-0.6mm}

\begin{minipage}[ht]{.296\linewidth}
\begin{tikzpicture}[x=5cm,y=5cm];
     \draw[red,thick] (0,0) -- (0.13,0.75);
     \draw[red,thick] (0.75,0.75) -- (0.62,0);
\foreach \i [count=\j from 0] in {35}{
   \draw [pattern=rotated hatch, pattern angle=\i] 
     ({mod(\j,4)}, {floor(\j/4)}) rectangle ++(0.75,0.75);   }
 \end{tikzpicture}
\end{minipage}\nolinebreak
\begin{minipage}[ht]{.2956\linewidth}
   \begin{tikzpicture}[x=5cm,y=5cm];
   \draw[thick] (0.75,0) --(0,0.75);

\node at (0.33,0.55) {\tiny $A'_2$};
\node at (0.52,0.35) {\tiny $A'_3$};
   \draw[red,thick] (0.44,0.31) --(0,0);
        \draw[red,thick]  (0.13,0) -- (0.495,0.255);
     \foreach \i [count=\j from 0] in {-10}{
   \draw [pattern=rotated hatch, pattern angle=\i] 
     ({mod(\j,4)} , {floor(\j/4)}) rectangle ++(0.75,0.75);   }
\end{tikzpicture}
\end{minipage}\nolinebreak

\vspace{-0.6mm}

\begin{minipage}[ht]{.296\linewidth}
\begin{tikzpicture}[x=5cm,y=5cm];
\draw[thick] (0.75,0) --(0,0.75);
\draw[red,thick] (0.31,0.44) --(0.75,0.75);

\node at (0.1,0.55) {\tiny $A_1$};
\node at (0.24,0.41) {\tiny $A_2$};
\node at (0.5,0.15) {\tiny $A_3$};

   \draw[red,thick] (0.255,0.495) -- ( 0.62,0.75);
\foreach \i [count=\j from 0] in {-10}{
   \draw [pattern=rotated hatch, pattern angle=\i] 
     ({mod(\j,4)}, {floor(\j/4)}) rectangle ++(0.75,0.75);   }
 \end{tikzpicture}
\end{minipage}\nolinebreak
\begin{minipage}[ht]{.2956\linewidth}
     \begin{tikzpicture}[x=5cm,y=5cm];
     \draw[red,thick] (0,0) -- (0.13,0.75);
      \foreach \i [count=\j from 0] in {35}{
   \draw [pattern=rotated hatch, pattern angle=\i] 
     ({mod(\j,4)} , {floor(\j/4)}) rectangle ++(0.75,0.75);   }
\end{tikzpicture}
\end{minipage}
\caption{The map $F: D \to D$, case $\phi_2 \in (\pi/4,\pi/2)$}
\label{fig-unfold2}
\end{figure}

 Another way to define $F$ is via unfolding,
this is shown in Figures \ref{fig-unfold1} and \ref{fig-unfold2}. 
The direction in the bottom left
square and in every other square  is $\theta_1$. The direction in the other squares is unfolded, 
thus the angle is $\pi -\theta_2$.  
For conveniences we use the  notation $\phi_1 = \theta_1, \phi_2 = \pi - \theta_2$.
We remark that the points in the interval $A_1$ have crossed the diagonal $D$ during the flow in the direction $\theta_2$, and arrive to $D$ with the same orientation; while the points in $A_2 \cup A_3$ do not cross $D$.  In the original chess billiard flow when they return to $D$ we need to apply
the central symmetry to define $F$, but this is not needed in the unfolded picture.

\begin{figure}[h]
\begin{minipage}{0.45\linewidth}
\begin{tikzpicture}[scale=1.5]

\draw [thick,decorate,decoration={brace,amplitude=10pt,mirror},xshift=0.4pt,yshift=-0.4pt](0,0) -- (1,0) node[black,midway,yshift=-0.6cm] {\tiny $A_2$};
\draw [thick,decorate,decoration={brace,amplitude=10pt,mirror},xshift=0.4pt,yshift=-0.4pt](1,0) -- (1.5,0) node[black,midway,yshift=-0.6cm] {\tiny $A_3$};
\draw [thick,decorate,decoration={brace,amplitude=10pt,mirror},xshift=0.4pt,yshift=-0.4pt](1.5,0) -- (3,0) node[black,midway,yshift=-0.6cm] {\tiny $A_1$};

\draw [thick,decorate,decoration={brace,amplitude=10pt},xshift=0.4pt,yshift=-0.4pt](0,2) -- (0,3) node[black,midway,xshift=-.6cm] {\tiny $A'_3 $};
\draw [thick,decorate,decoration={brace,amplitude=10pt},xshift=0.4pt,yshift=-0.4pt](0,1.5) -- (0,2) node[black,midway,xshift=-.6cm] {\tiny $A'_2 $};
\draw [thick,decorate,decoration={brace,amplitude=10pt},xshift=0.4pt,yshift=-0.4pt](0,0) -- (0,1.5) node[black,midway,xshift=-.6cm] {\tiny $A'_1$};

\draw[] (3,0) -- (0,0) -- (0,3);

\draw[dotted] (1,0) -- (1,3);
\draw[dotted] (0,3) -- (3,0);
\draw[dotted] (1.5,0) -- (1.5,3);
\draw[dotted] (3,0) -- (3,3);
\draw[dotted] (0,1.5) -- (3,1.5);
\draw[dotted] (0,2) -- (3,2);
\draw[dotted] (0,3) -- (3,3);

\draw[] (1,2) -- (1.5,3);
\draw[]  (0,1.5) -- (1,2);
\draw[] (1.5,0) -- (3,1.5);
\end{tikzpicture}
\end{minipage}\qquad \nolinebreak
\begin{minipage}{0.45\linewidth}
\begin{tikzpicture}[scale=1.5]

\draw [thick,decorate,decoration={brace,amplitude=10pt,mirror},xshift=0.4pt,yshift=-0.4pt](0,0) -- (1,0) node[black,midway,yshift=-0.6cm] {\tiny $A_1$};
\draw [thick,decorate,decoration={brace,amplitude=10pt,mirror},xshift=0.4pt,yshift=-0.4pt](1,0) -- (1.5,0) node[black,midway,yshift=-0.6cm] {\tiny $A_2$};
\draw [thick,decorate,decoration={brace,amplitude=10pt,mirror},xshift=0.4pt,yshift=-0.4pt](1.5,0) -- (3,0) node[black,midway,yshift=-0.6cm] {\tiny $A_3$};

\draw [thick,decorate,decoration={brace,amplitude=10pt},xshift=0.4pt,yshift=-0.4pt](0,2) -- (0,3) node[black,midway,xshift=-.6cm] {\tiny $A'_1 $};
\draw [thick,decorate,decoration={brace,amplitude=10pt},xshift=0.4pt,yshift=-0.4pt](0,1.5) -- (0,2) node[black,midway,xshift=-.6cm] {\tiny $A'_3 $};
\draw [thick,decorate,decoration={brace,amplitude=10pt},xshift=0.4pt,yshift=-0.4pt](0,0) -- (0,1.5) node[black,midway,xshift=-.6cm] {\tiny $A'_2$};

\draw[] (3,0) -- (0,0) -- (0,3);

\draw[dotted] (1,0) -- (1,3);
\draw[dotted] (0,3) -- (3,0);
\draw[dotted] (1.5,0) -- (1.5,3);
\draw[dotted] (3,0) -- (3,3);
\draw[dotted] (0,1.5) -- (3,1.5);
\draw[dotted] (0,2) -- (3,2);
\draw[dotted] (0,3) -- (3,3);

\draw[] (0,2) -- (1,3);
\draw[] (1,0) -- (1.5,1.5);
\draw[] (1.5,1.5) -- (3,2);
\end{tikzpicture}
\end{minipage}\nolinebreak
\caption{The map $F$ for $\phi_2 \in (0,\pi/4)$ and $(\pi/4,\pi/2)$.}\label{fig-F}
\end{figure}

 The graph of the map $F: D \to D$ has two possible forms, they are
 shown in Figure \ref{fig-F}.  
The set $D$ decomposes into three segments $A_1,A_2,A_3$ such that the derivative $F'|_{A_i}$ is constant for each $i$; we call their images $A'_i = F(A_i)$.
 Let $a_i := |A_i|$, where $| \cdot |$ denotes the length of a segment.  In the case $\pi/4 < \phi_2 < \pi/2$ 
the central
symmetry of Figure \ref{fig-unfold2} about the point $(1/2,3/2)$ implies $|A'_1| = |A_1|$;
while the central symmetry of the figure about the point $(1,1)$ yields $|A'_2|=|A_3|$; $|A'_3|=|A_2|$ and thus 
$$\frac{|A'_2|}{|A_2|} = \frac{|A_3|}{|A'_3|}.$$
Notice that these symmetries imply that the point $(a_1 + a_2, F(a_1 + a_2))$ of the graph of $F$ lies on the anti-diagonal marked in dots in Figure \ref{fig-F}, i.e.,
$a_1 + a_2 + F(a_1+a_2) = 1$. (Similar symmetries arise in the case $\phi_2 \in (0,\pi/4)$).
 
The length of $D$ is $\sqrt{2}$/ We parametrize $D$ with arclength and note that $F(0) = F(\sqrt{2})$, thus we
think of  $D$ as a circle of length $\sqrt{2}$ which we do not normalize. 
Elementary plane geometry  (see Figures \ref{fig-unfold1} and \ref{fig-unfold2}) yields
\begin{align*}
a_1 & =  \big  ( 1 - \tan(\phi_2) \big ) \frac{\sin(\pi/2 - \phi_1)}{\sin(\pi/4 + \phi_1)} & \text{ if }  & 0 < \phi_2 < \frac{\pi}{4}\\
a_1 &= \big  ( 1 - \cot(\phi_2) \big ) \frac{\sin(\phi_1)}{\sin(3\pi/4 - \phi_1)} & \text{ if } &  \frac{\pi}{4} < \phi_2 < \frac{\pi}{2}.
\end{align*}

Remark:  if $\phi_2 = \pi/4$ the interval $A_1$ disappears, and there are only two intervals; on the other hand
if $\phi_1 = \phi_2$, then $F$ is a circle rotation by the length $a_1 \pmod{\sqrt{2}}$.

It is not hard to check that if we increase $\phi_2$ (i.e., decrease $\theta_2$) then the graphs of the resulting maps 
$F = F_{\phi_1,\phi_2}$ and $F_{h} = F_{\phi_1,\phi_2+ h}$ do not intersect
(see Figure \ref{fig-FF}).

\begin{figure}[h]
\begin{tikzpicture}[scale=1.5]

\draw[] (3,0) -- (0,0) -- (0,3);
\draw[dotted] (0,3) -- (3,0);
\draw[dotted] (3,0) -- (3,3);
\draw[dotted] (0,3) -- (3,3);

\draw[] (0,2) -- (1,3);
\draw[] (1,0) -- (1.5,1.5);h
\draw[] (1.5,1.5) -- (3,2);

\draw[red] (0,1.9) -- (1.1,3);$$ 
\draw[red] (1.1,0) -- (1.65,1.35);
\draw[red] (1.65,1.35) -- (3,1.9);
\end{tikzpicture}
\caption{The original map is in black, and the red map arises by increasing $\phi_2$ a bit.}\label{fig-FF}
\end{figure}
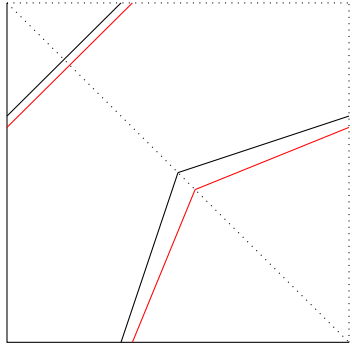

 \begin{lemma}\label{ll2} If $\phi_1 \in (0,\pi/2)$ and  $z \in D$,  then 
   $$\frac{\partial{F(z)}}{\partial{\phi_2}}  \le    \frac{- \sin(\phi_1)}{\sin(3\pi/4 - \phi_1)  } .$$
 \end{lemma}
 
\begin{proof} The constant $a_1$ defined above varies with $h$, we denote this dependence by $a_1(h)$.  
Suppose $h$ is such that $\phi_2 + h \in (\pi/4,\pi/2)$, then
\begin{align*}
  F_{h}(0) - F(0)  & =  a_1 -  a_1(h) \\
& =   \Big ( \cot(\phi_2 + h )  - \cot(\phi_2 )   \Big)  \frac{\sin(\phi_1)}{\sin(3\pi/4 - \phi_1)},
\end{align*}
and thus 
\begin{align*} \frac{ \partial F_{\phi_1,\phi_2}}{\partial \phi_2}(0) & = \lim_{h \to 0} 
\frac{  \cot(\phi_2 + h )  - \cot(\phi_2 )  }{h} \frac{\sin(\phi_1)}{\sin(3\pi/4 - \phi_1)}\\
& =- \csc^2{(\phi_2)} \frac{\sin(\phi_1)}{\sin(3\pi/4 - \phi_1)}\\
& \le \frac{- \sin(\phi_1)}{\sin(3\pi/4 - \phi_1)}.
\end{align*}
However
\begin{align*}
F_{h}(z) - F(z)  & \le F_{h}(0) - F(0) \text{ for } h > 0\\
F_{h}(z) - F(z)  &\ge F_{h}(0) - F(0) \text{ for } h < 0
\end{align*}  
for any $z \in D$, and thus
$$\frac{ \partial F_{\phi_1,\phi_2}}{\partial \phi_2}(z)  \le \frac{ \partial F_{\phi_1,\phi_2}}{\partial \phi_2}(0).$$
\end{proof}

\begin{proof}
{\bf Denseness:}
Suppose  that the rotation number of $F$ is irrational.  Fix $\e > 0$ satisfying $\e < \min(\pi/2-\phi_2, \phi_2 - \pi/4)  < \sqrt{2}/2$. 
 Choose a point $z \in D$ such that $z \in \omega(z)$ (the $\omega$-limit set of the orbit $z$) and fix an  $n > 0 $ so that $|F^n(z) - z|< \varepsilon$.

Remember that $D$ is a circle, we define $z_2-_Dz_1$ the signed distance between points 
$z_1,z_2 \in D$  as follows. If $0 \le z_1 \le z_2 \le \sqrt{2}$ and $z_2-z_1 < \sqrt{2}/2$
then $z_2 -_D z_1 = z_2-z_1$  while if  $0 \le z_2 \le z_1 \le  \sqrt{2}$ 
and $\sqrt{2}-(z_2-z_1) < \sqrt{2}/2$ the $z_2 -_D z_1 =  \sqrt{2}-(z_2-z_1) < \sqrt{2}/2$ and then extend to negative
distances by setting $z_1 -_D z_2 = -(z_2 -_D z_1)$. Throughout the rest of this section we will simply write $-$ instead of $-_D$.

Let $ (h^{(0)}_0,h^{(0)}_1)$ be the maximal interval containing $0$ such that the points $z$ and $F^n(z)$ are in the same
semicircle,  which allows us to define the continuous
function $e : (h^{(0)}_0,h^{(0)}_1) \to \R$ by
$e(h) :=  F^n_h(z) - z$. 

Similarly let $(h^{(i)}_0,h^{(i)}_1)$ ($i=1,2$) be the maximal intervals containing $0$  where the functions
$P_1(h) :=   F^n_h(z) - F_h(F^{n-1}(z)) $ and  $P_2(h) := F_h(F^{n-1}(z)) - F^n(z) $ are respectively  defined.

Consider the interval 
$$ 
(h^{(3)}_0,h^{(3)}_1) :=  \bigcap_{i=0}^2(h^{(i)}_0,h^{(i)}_1).
$$ 
Note that $(h^{(3)}_0,h^{(3)}_1)$ contains the points $0$ and depends on $\phi_1,\phi_2$ and on $z$ which are fixed throughout the proof but does not depend on the choice of $\e$. For all $h \in (h^{(3)}_0,h^{(3)}_1)$, we have
$$
e(h)  = P_1(h) + P_2(h)  + (F^n(z) - z).
$$ 

We need to estimate each of these terms.
We have already supposed that  $|  (F^n(z) - z) | < \e$.

Note that $P_1(0) = 0$.
The function $F^{n-1}_h(z)$ is  a decreasing function of $h$, thus 
since 
 $F_h$ is an increasing function of $z$, one obtains 
 \begin{align*}
 P_1(h)  & \le 0 \text{ if } h \ge 0\\ 
 P_1(h)  & \ge 0 \text{ if } h \le 0.
 \end{align*}

We also have $P_2(0) = 0$,  and  applying the Lemma yields
\begin{align*}
P_2(h)  & \le Ch \text{ for } h \ge 0\\
P_2(h)  &\ge -Ch \text{ for } h  \le  0
\end{align*}  
 where $C :=  \frac{- \sin(\phi_1)}{\sin(3\pi/4 - \phi_1)  } $
is a negative constant.
 
First suppose $F^n(z) - z$ is positive. The functions $P_1$ and $P_2$ are both  
 negative and continuous on the open interval  $(0,h_1^{(3)})$. Furthermore 
 $P_1(h_1^{(3)}) > 0$ and $P_2(h_1^{(3)})>0$.  Since  $h^{(3)}_1$ does not depend on $\e$,
 it follows that if $0 < \e <  P_1(h_1^{(3)})  + P_2(h_1^{(3)}) $, then there is an  $h' \in (0, h^{(3)}_1)$
 such that $e(h') = 0$, and thus $F^n_{h'}(z) = z$.
 
 The case $ F^n(z) - z < 0$ is similar, varying $h \in (h_0^{(3)},0)$.
  
Choosing $\e > 0$ arbitrarily small and remembering that  $h^{(3)}_0$ and $h^{(3)}_0$ do not depend on $\e$, yields   $h'$  arbitrarily close to $0$; showing that the periodic directions are dense.\\

 \noindent
{\bf Openess:}
Suppose that 
$z_0$  is a periodic point of period $n$ for the map $F_{\phi_1,\phi_2}$ where the directions additionally satisfy $\phi_1 \not \in \{0,\pi/2\}$
and $\phi_2 \not \in \{ 0,\pi/4,\pi/2\}$.
The graph of $y = F^n(z)$ intersects the diagonal at the  point $(z_0,z_0)$,
if this intersection is transverse then
since the graph of $F$ (and thus also the graph of $F^n$) changes continuously with  $(\phi_1,\phi_2)$ (and thus with $(\theta_1,\theta_2)$) the intersection persists for a non-empty open set of parameters. 

Now suppose
that the intersection is not transverse, then either 
(i)  $(F^n)'(z_0)$ does not exist and thus the orbit of the periodic point $z_0 = F^n(z_0)$ must pass through a corner of the polygon;
or (ii)  $(F^n)'(z_0)=1$, in this case there is an interval   $J := ( z_- , z_+)$ containing $z_0$ such 
that $F^n|_{J}$ is the identity map.

Suppose that the graph of $F^n$ stays below the diagonal except for the tangency at the point $(z_0,z_0)$, respectively on the
segment $\{(z,z): z \in J\}$.  Then since $F^n_{h}$ is decreasing,
for all sufficiently small negative $h$ the graph of $F^n_{h}$ will cross the diagonal transversely at a point near $(z_0,z_0)$, respectively at two ponts near $(z_-,z_-)$ and $(z_+,z_+)$. The case when the graph of $F^n$ is above the diagonal is treated similarly using positive $h$.
 \end{proof}

\begin{proposition}\label{prop-connection1}  In the square for  a non-exceptional direction
an orbit passing through 
 a vertex is periodic if and only if it is a connection.
 \end{proposition}
 
 \begin{proof}
Proposition \ref{prop-connection2} yields the converse assertion of the lemma since
in the square connections must connect vertices.

Vertices which are fixed points are connections.  Now consider the case when the period of an orbit of a vertex $a$ is at least 2 and thus by Proposition \ref{prop-fix} the directions must be in different quadrants; so one of
 the directions is a supporting line at $a$. Thus $a$ acts as a reflector  in the sense that after hitting this corner the orbit 
retraces itself  backwards. 
The orbit going through the corner $a$ is periodic (and is not a fixed point), thus
it must make its way back to  $a$. Since
this is the only mechanism for retracing  an orbit, the only way this can happen is by
retracing the orbit once again:
the orbit must hit a different corner, i.e., it is a connection.\end{proof}

\begin{proposition}\label{prop-connection} In the square, 
if there is  a neutral periodic orbit in a non-exceptional direction, then there is a connection in this direction.
\end{proposition}

\begin{proof}
If $\theta_1$ and $\theta_2$ are parallel, then all points are fixed by $S$, thus each side of $P$ is a neutral cylinder and each vertex of $P$ is a connection.

Suppose now that $\theta_1$ and $\theta_2$ are not parallel and $(\theta_1,\theta_2)$ is non-exceptional.
Let $q$ denote the period of the neutral cylinder and
consider a maximal interval $I$ defining the neutral periodic cylinder. 
By definition
 $C(I) = C(\S(I)) = \cdots = C(\S^{q-1}(I))$, thus we can choose $0 \le j < q$ such that
$\S^j(I) = [a,b]$, where $a$ is a 
vertex  of the polygon, otherwise we could extend $I$ to a larger interval.
Since $C(I)$ is closed, the orbit of the vertex $a$ is periodic and thus a connection by Proposition \ref{prop-connection1}.
\end{proof}

The proof shows a bit more.  If we consider the other side of the cylinder it also passes through a 
vertex, and repeating the proof shows that the orbit of this vertex is also a connection.  Thus either
there are two connections, or a single connection which bounds both sides of the cylinder.

\begin{proposition}\label{prop-tre}
For the square, the set $$\Big\{(\theta_1,\theta_2): \S  \text{ has a neutral periodic orbit} \Big \}$$ is  a union  of at most countably many
one-dimensional sets.
\end{proposition}

\begin{proof} Using the previous Proposition, it suffices to prove that the set 
$$\Big\{(\theta_1,\theta_2): \S  \text{ has a coonection in this direction} \Big \}$$ is  a union  of at most countably many
one-dimensional sets.

We will use the following implication of a strengthening of the implicit function theorem (IFT) due to Kumagai \cite{Ku}:\\
\textit{ Consider a continuous function $f:\R\times \R\to \R$ and a point $(\theta_1^0,\theta_2^0)$ such that  $f(\theta_1^0,\theta_2^0)=0$. 
If there exist open neighborhoods $C\subset \R$ and $E\subset \R$ of $\theta_1^0$ and $\theta_2^0$, respectively, such that, for all $\theta_1 \in C$, 
$f(\theta_1 ,\cdot): E\to \R$ is locally one-to-one then there exist open neighborhoods $C_{0}\subset \R$ and $E_{0}\subset \R$ of $\theta_1^0$ and $\theta_2^0$, 
such that, for every $\theta_1 \in C_{0}$, the equation $f(\theta_1, \theta_2) = 0$ has a unique solution
$\theta_2=g(\theta_1)\in E_{0}$,
where $g$ is a continuous function from $C_0$ into $E_0$.}

Consider a neutral periodic orbit in the direction $(\theta_1^0,\theta_2^0)$, and one of the associated connections given by the Proposition \ref{prop-connection}.  Suppose that this saddle connection starts at a vertex $a$.
We use the representation $F:D \to D$ given in Theorem \ref{thm-sq} and by a slight misuse of 
notation we will also denote the point in $D$ on this saddle connection by $a$, so 
$a =F^n_{\theta_1^0,\theta_2^0} (a)$. This point depends on $\theta_1$, 
in the proof $\theta_1$ is fixed, and $\theta_2$ varies, thus the identification of the vertex $a$ and
with a point in the diagonal remains valid throughout the proof.

Consider a lift $\tilde{F}: \R \to \R$ of $F$.  Then
there is an $m \in \Z$ such that 
$a + m  =\tilde{F}^n_{\theta_1^0,\theta_2^0} (a)$.
 The proposition follows immediately if we can apply
Kumagai's IFT to the function
$$f(\theta_1,\theta_2) := \tilde{F}^n_{\theta_1,\theta_2} (a) - (a+m).$$

The Proof of Theorem \ref{thm-sq}  shows that there is an interval $E$ such that for each $\theta_1$
the function $F^n(\theta_1,\cdot)$ is  a strictly monotonic map of $\theta_2 \in E$. This immediately implies that $F^n(\theta_1,\cdot)$ is a strictly monotonic 
map of $\theta_2 \in E$, and thus so is $\tilde{F}^n(\theta_1,\cdot)$.
Thus $f(\theta_1,\cdot)|_E$ is locally one to one and we can apply Kumagai's theorem.
\end{proof}

\section{Centrally symmetric domains.}\label{sec6}

\begin{proposition}
Suppose that $P$ is centrally symmetric, for example a circular or square table and suppose that $\theta_1,\theta_2$ are such that the rotation number of
$\S$ is irrational, then for any $x \in Q_i$ the $\omega$-limit set  $\omega(x) \subset Q_i$ is centrally symmetric
$(i=1,2)$.
\end{proposition}

\begin{proof}
Consider two points $x^{\pm} \in \partial P$ such that the vector ${x^- x^+}$ passes through the center of symmetry
of $P$ and is in the direction $\theta_1$.
The $\omega$-limit sets of these two points are centrally symmetric to each other, however
from Corollary \ref{cor-rotation} we have $\omega(x)$ does not depend on $x$.   
\end{proof}

\end{document}